\documentclass[12pt]{article}
\usepackage{latexsym}
\usepackage{amsmath,amsfonts,amssymb,mathrsfs,graphicx}
\usepackage{amsthm}
\usepackage{graphicx}
\usepackage{verbatim}
\usepackage{cite}
\usepackage{hyperref}
\usepackage{color} 
\usepackage{lineno}
\usepackage{scalerel,stackengine}
\stackMath

\title{Frostman and Fourier characterisations of fractal dimensions}
\author{Kenneth J. Falconer and Shuqin Zhang}
\date{}

\def\bbbr{\mathbb{R}}

\newcommand{\D}{\mathcal{D}}

\newcommand{\e}{\epsilon}
\newcommand{\Nearrow}{\rotatebox[origin=c]{20}{\(\leq\)}}
\newcommand{\Searrow}{\rotatebox[origin=c]{340}{\(\leq\)}}
\DeclareMathOperator{\spt}{spt}

\newcommand\reallywidehat[1]{%
\savestack{\tmpbox}{\stretchto{%
  \scaleto{%
    \scalerel*[\widthof{\ensuremath{#1}}]{\kern.1pt\mathchar"0362\kern.1pt}%
    {\rule{0ex}{\textheight}}
  }{\textheight}%
}{2.4ex}}%
\stackon[-6.9pt]{#1}{\tmpbox}%
}

\newcommand\ucod{\overline{\mbox{\rm dim}}_{\rm C}} 
\newcommand\lcod{\underline{\mbox{\rm dim}}_{\rm C}} 
\newcommand\ubd{\overline{\mbox{\rm dim}}_{\rm B}} 
\newcommand\lbd{\underline{\mbox{\rm dim}}_{\rm B}} 
\newcommand\bdd{\mbox{\rm dim}_{\rm B}}
\newcommand\pkd{\mbox{\rm dim}_{\rm P}} 
\newcommand\hdd{\mbox{\rm dim}_{\rm H}} 
\newcommand\mlbd{\underline{\mbox{\rm dim}}_{\rm MB}} 
\newcommand\mubd{\overline{\mbox{\rm dim}}_{\rm MB}} 

\allowdisplaybreaks

\newcommand{\be}{\begin{equation}} 
\newcommand{\ee}{\end{equation}} 

\newcommand{\R}{\mathbb{R}}

 \newtheorem{theo}{Theorem}[section]
 \newtheorem{cor}[theo]{Corollary}
 \newtheorem{lem}[theo]{Lemma}
 \newtheorem{prop}[theo]{Proposition}

\begin{document}
\maketitle

\begin{abstract}
We examine Frostman-type characterisations and other extremal measure criteria for a range of fractal dimensions of sets. In particular we derive properties of the less familiar modified lower box dimension and upper correlation dimension. We also express a number of fractal dimensions in terms of Fourier properties of measures.
\end{abstract}

\section{Introduction}
\setcounter{equation}{0}
\setcounter{theo}{0}

\subsection{Overview}
Starting with the introduction of Hausdorff measure and dimension in 1919, many ways of assigning a dimension to fractal sets have been introduced. Almost all of these fractal dimensions are intimately related to measures, in particular the dimension of a set $E \subset \mathbb{R}^d$  can be characterised as the extreme value of some expression  taken over probability measures supported by $E$.

The prototype for this goes back to Frostman's work in the 1930s \cite{Fro}. He showed that there exists a Borel probability measure $\mu$ supported by a closed (or more generally Borel) set $E$ such that 
\begin{equation}\label{frost}
\mu(B(x,r))\leq cr^s \quad \text{ for all } x \in \R^d \text{ and } r>0
\end{equation}
for some constant $c>0$ if and only if 
$\mathcal{H}^s(E)>0$, where $\mathcal{H}^s$  is $s$-dimensional Hausdorff measure.
In particular this leads easily to a characterisation of Hausdorff dimension $\hdd$:
\begin{equation}\label{frostdim}
\hdd E 
	=  \sup_{\mu\in {\mathcal P}(E)} \liminf_{r\to 0} \inf_{x\in E} \frac{\log  \mu(B(x,r))}{ \log r},\end{equation}
	where  ${\mathcal P}(E) $ denotes the set of Borel probability measures supported by $E$. 

The Frostman inequality \eqref{frost} is closely related to an integral form:
\begin{equation}\label{intballs}
 \int \mu(B(x,r))\mathrm{d}\mu(x) \leq cr^s  \text{ for all }  r>0.
\end{equation}
(Note that the integral here equals $(\mu\times\mu)\{(x,y) \in E \times E :|x-y|\leq r\}$.)
Clearly \eqref{frost} immediately implies \eqref{intballs} for all $r$, but also it is easily shown that \eqref{intballs} implies a Frostman inequality for a restriction of the measure $\mu$. 

There are other well-known characterisations of Hausdorff dimension by integral forms, see Section \ref{lcdhd}, as well as Fourier transformed versions of some of these criteria in Section \ref{FourierSec}.

A benefit of having a range of measure characterisations of the dimension of a set is that different forms can be appropriate for different applications.  In this paper we present measure representations for a number of fractal dimensions, some are known and are included for context and overview. Indeed, Cutler \cite{Cu} gives a thorough treatment for Hausdorff dimension $\hdd$ and packing dimension $\pkd$. Here we are specially interested in the rather less studied modified lower box dimension $\mlbd$ and upper correlation dimension $\ucod$. In particular we obtain a Frostman-type expression for $\mlbd$ and we show that $\mlbd$, $\ucod$ and $\pkd$ are essentially different, providing a counter-example to a conjecture of Fraser. In the final section we establish an equivalence that enables many of the criteria to be expressed in terms of Fourier transforms.

\subsection{Notation and conventions}
Throughout this paper we will work with subsets of $d$-dimensional Euclidean space $\R^d$, though much extends to more general metric spaces. We write $B(x,r)$ for the closed ball of centre $x$ and radius $r$.  We will assume that all sets and measures mentioned are Borel though this will not always be stated explicitly.  For a set $E$ an overbar $\overline{E}$ denotes the closure, ${\rm int}E$ the interior, and $\partial E$ the boundary. We write $f\asymp g$ to mean that there are constants $c_1,c_2 >0$ such that 
$c_1 f(r) \leq g(r) \leq c_2 f(r)$ over an appropriate range of $r$.

We define a dyadic cube in $\R^d$ with side $2^{-n} (n\geq0)$  as the $d$-fold product of half-open intervals $(j_12^{-n},(j_1+1)2^{-n}]\times\cdots\times(j_d2^{-n},(j_d+1)2^{-n}]$, where $j_k\in\mathbb{Z}$ for  $k=1,\cdots,d$.
	We will frequently refer to a dyadic cube of side $2^{-n}$   as an $n$th-level cube  and denote such a cube by $C_n$. For every $n\geq0$, each $x\in\R^d$ is contained in exactly one $n$th cube, which we denote  by $C_n(x)$. In much of what follows, results involving expressions such as ${\log  \mu(B(x,r))}/{ \log r}$ are equally valid where this is replaced by the dyadic form ${\log  \mu(C_n)}/{ -n\log 2}$, using that balls are contained in cubes of comparable radii and vice-versa.  However, working with  dyadic versions is often easier because of the `nested' property of dyadic cubes, that any two cubes are either disjoint or one contains the other.

\section{Frostman-type conditions for dimensions}\label{sec2}
\setcounter{equation}{0}
\setcounter{theo}{0}
In this section we consider how various definitions of fractal dimensions of a set $E$ may be characterised as maxima of expressions involving  measures supported by $E$.

\subsection{Correlation dimensions}

	We  first consider correlation dimensions of measures before transferring the definition to sets. 
	Let $\mu$ be a Borel probability  measure on $\R^d$.
	The {\it lower correlation dimension}  of $\mu$ is defined by 
	\begin{equation}\label{lcdm}
		\lcod \mu=\liminf_{r\rightarrow0}\dfrac{\log\int\mu(B(x,r))\,\mathrm{d} \mu(x)}{\log r}, 
	\end{equation}
and similarly the  {\it upper correlation dimension}  of $\mu$ is given by 
	\begin{equation}\label{ucdm}
		\ucod \mu=\limsup_{r\rightarrow0}\dfrac{\log\int\mu(B(x,r))\,\mathrm{d} \mu(x)}{\log r}. 
	\end{equation}	
If the limit exists, that is $\lcod \mu=\ucod \mu$, we term the common value the {\it correlation dimension} of $\mu$; this is the case for many common measures including self-similar and self-conformal measures. Note that correlation dimension has many alternative names across the literature including information dimension, $L^2$-dimension and energy dimension.
Correlation dimensions of measures have been intensively studied, not least as a special case of $L^q$-dimensions in multifractal analysis and also in information theory. The paper by Mattila, Moran and Rey \cite{MM} presents a number of properties of correlation dimensions of measures that show they are not particularly well-behaved. Projection theorems for  lower correlation dimensions of measures are  given in \cite{FO} and \cite{HK} and for upper correlation dimensions   in \cite{FO}, along with results for the more general $L^q$-dimensions. See also \cite{Cu, FFK}.
	
The correlation dimensions of sets are now defined in terms of those of measures in the natural way.
For a Borel set $E \subset \R^d$ the {\it lower correlation dimension} is given by
\begin{equation}\label{lcds}
	\lcod E=\sup_{\mu\in\mathcal{P}(E)}\lcod \mu,
\end{equation}
and the {\it upper correlation dimension} by
\begin{equation}\label{ucds}
	\ucod E=\sup_{\mu\in\mathcal{P}(E)}\ucod \mu,
\end{equation}
and again we refer to the {\it correlation dimension} of $E$ if these are equal.

\subsubsection{Lower correlation dimension -- Hausdorff dimension}\label{lcdhd}

The lower correlation dimension of a set equals its Hausdorff dimension.
For the definition of Hausdorff dimension first recall that a $\delta$-{\it cover } of $E\subset\R^d$ is a finite or countable collection of sets $\{U_i\}_i$, each with diameter at most $\delta$, that covers $E$. That is, $E \subset \bigcup_i U_i$ and $0<|U_i|\leq \delta$, where $|\ |$ denotes diameter.  For each $s\geq0$ and $\delta>0$, let
\begin{equation*}
	\mathcal{H}_\delta^s(E)=\inf\Big\{ \sum_i |U_i|^s:\{U_i\}_i \text{ is a } \delta\text{-cover of }E\Big\}.
\end{equation*}
Since $\mathcal{H}_\delta^s(E)$ increases as $\delta\rightarrow0$, the limit 
\begin{equation*}
	\mathcal{H}^s(E)=\lim_{\delta\rightarrow0}\mathcal{H}_\delta^s(E)
\end{equation*}
exists, possibly infinite. We term $\mathcal{H}^s(E)$ the {\it $s$-dimensional Hausdorff measure} of $E$. Then $\mathcal{H}^s$ is a Borel measure on $\R^d$.
The {\it Hausdorff dimension} of $E$ is 
\begin{equation*}
	\hdd E=\inf\left\lbrace s\geq0:\mathcal{H}^s(E)=0\right\rbrace=\sup\left\lbrace s:\mathcal{H}^s(E)=\infty\right\rbrace.
\end{equation*}

We have already noted Frostman's characterisation of Hausdorff dimension in \eqref{frost} and \eqref{frostdim}. To see that Hausdorff dimension equals lower correlation dimension, 
first note that if $s$ is such that \eqref{frost} is satisfied then integrating gives $\int\mu(B(x,r))\,\mathrm{d} \mu(x) \leq c r^s$, from which if follows that $\hdd E  \leq \lcod E$. On the other hand, 
 if  there is a measure $\mu\in \mathcal{P}(E)$ such that $\lcod \mu >s$ then there exists $c>0$ such that $\int\mu(B(x,r))\,\mathrm{d} \mu(x) \leq c r^s$ for $0<r\leq 1$, so for all $\epsilon >0$ 
\begin{equation}\label{infsum}
\int \Big(\sum_{k=1}^\infty 2^{-k(\epsilon-s)}\mu\big(B(x,2^{-k})\big)\Big)\mathrm{d}\mu(x) \leq c\sum_{k=1}^\infty 2^{-k\epsilon}  <\infty.
\end{equation}
It follows that  there is a set $E' \subset E$ with $\mu (E')>0$ such that $\mu(B(x,2^{-k}))\leq c'2^{-k(s-\epsilon)}$ for all $x \in E'$.
 
This  inequality extends from dyadic to continuous values of $r$ to give a Frostman inequality, so \eqref{frostdim} and \eqref{lcds}  imply that $\hdd E \geq s-\epsilon$ and 
\begin{equation}\label{haucor}
 \hdd E = \lcod E.
 \end{equation}

Closely related is the energy characterisation of Hausdorff dimension by which the bounds at all scales can be combined into a single expression. 
The {\it $s$-energy } of a measure $\mu$ is given by 
\begin{equation}\label{energy}
 	\mathcal{E}_s(\mu)=\iint\frac{\mathrm{d}\mu(x)\mathrm{d}\mu(y)}{|x-y|^{s}}
 \end{equation}
Then
$$ \hdd E = \lcod E = \sup\Big\{ s\geq0: \exists\  \mu\in \mathcal{P}(E)\text{ s.t. }\mathcal{E}_s(\mu)<\infty\Big\},$$ 
see, for example, \cite{FaBk,Mat}.
	In fact,  if $ \mathcal{E}_s(\mu)<\infty$ for some $s>0$, it follows from 
	\begin{equation*}
		\mathcal{E}_s(\mu)\geq\iint_{|x-y|\leq r}\frac{1}{|x-y|^{s}}\mathrm{d}\mu(y)\mathrm{d}\mu(x)\geq \frac{\int\mu(B(x,r))\mathrm{d}\mu(x)}{r^s}
	\end{equation*}
	that $\lcod E\geq s$. On the other hand, if $t$ is such that \eqref{frost} is satisfied, then for any $0<s<t$,
	\begin{equation*}
		\mathcal{E}_s(\mu)=\int\int_0^\infty  sr^{-s-1}\mu(B(x,r))\mathrm{d}r\mathrm{d}\mu(x)
		\leq cs\int_0^\infty  r^{t-s-1}\mathrm{d}r<\infty.
	\end{equation*}

\subsubsection{Upper correlation dimension}
The upper correlation dimensions of measures \eqref{ucdm} and sets \eqref{ucds} are less studied than their lower counterparts. There are various useful equivalent definitions of upper correlation dimension.
\begin{prop}\label{cordim}
Let $E \subset \mathbb{R}^d$ be a Borel set. Then
\begin{align}
&\ucod E\nonumber\\
&= \inf\Big\{s \geq 0:  \forall\ \mu \in {\mathcal P}(E), \text{ for all sufficiently small }r, \int \mu(B(x,r))\mathrm{d}\mu(x) \geq r^{s } \Big\} \label{cd1}\\
&= \sup\Big\{s \geq 0: \exists\ \mu \in {\mathcal P}(E)  \text{ and $\{r_k\}_k\searrow 0 $ s.t.  } \int \mu(B(x,r_k))\mathrm{d}\mu(x) \leq r_k^{s } \Big\} \label{cd2} \\
&= \sup\Big\{s \geq 0: \exists\   \mu \in {\mathcal P}(E)  \text{ and $\{r_k\}_k\searrow 0 $ s.t.  } \mu(B(x,r_k))\leq r_k^{s } \text{ for all } x \in \mathbb{R}^d\Big\}\label{cd3}\\
&= \sup\Big\{s \geq 0: \exists\   \mu \in {\mathcal P}(E)  \text{ and $\{n_k\}_k\nearrow\infty $ s.t.  } \mu(C_{n_k})\leq 2^{-n_ks}  \text{ for all } C_{n_k}  \in \D_{n_k}\Big\}\label{cd4}.
\end{align}
\end{prop}

\begin{proof}
Identity \eqref{cd1} follows directly from the definitions \eqref{ucdm} and \eqref{ucds}. The equivalence between \eqref{cd1} and \eqref{cd2}, as well as that between \eqref{cd3} and \eqref{cd4}, is clear, and \eqref{cd3} clearly implies \eqref{cd2}. To see that \eqref{cd2} implies \eqref{cd3}, assume without loss of generality that  $r_k$ decreases at least geometrically, 
and integrate the sum of terms $r_k^{\epsilon - s }\mu(B(x,r_k))$ as was done for \eqref{infsum}.
 \end{proof}

We remark that, according to \eqref{cd3}, it is easy to obtain 
\begin{equation*}
	\displaystyle{\ucod } E=\sup_{\mu\in {\mathcal P}(E) }\limsup_{r\rightarrow0}\inf_{x\in E}\frac{\log\mu(B(x,r))}{\log r}.
\end{equation*}
Moreover, it is evident that $\ucod\mu\geq\dim_{\rm F}\mu$, where the {\it Frostman dimension} of $\mu$ is defined as 
\begin{align}
	\dim_F \mu= \sup &\big\{s  \geq 0: \text{{\rm   there exists a   constant $c>0$  such that }} \mu(B(x,r))\leq cr^{s } \text{\rm \  for all}\notag\\
&\qquad \text{{\rm  $x \in \mathbb{R}^d$ and $0<r<1$\big\}}}\notag\\
&=\liminf_{r\to 0} \inf_{x\in {\R^d}} \frac{\log  \mu(B(x,r))}{ \log r}.
\end{align}

There is no general relationship between the upper correlation dimension  $\displaystyle{\ucod }\mu $ of $\mu$ and the lower Minkowski dimension $\underline{\dim}_{\rm M}\mu$ of $\mu$ defined in \cite{FFK} as
\begin{equation*}
	\underline{\dim}_{\rm M}\mu:=\liminf_{r\to 0} \sup_{x\in {\R^d}}\frac{\log  \mu(B(x,r))}{ \log r}.
\end{equation*}

In Section \ref{sec3} we consider the relationship between $\ucod E$ and other dimensions.

\subsection{Box-counting dimensions}

For $r>0$ and $E\subset \mathbb{R}^d$ bounded and non-empty we write $N_r(E)$ for the least number of sets of diameter $r$ that can cover $E$. 
The {\it lower} and {\it upper box-counting dimensions} or {\it Minkowski dimensions} of   $E $ are given by 
\begin{equation}\label{bdimdef}
\lbd E\ =\ \liminf_{r\to 0} \frac{\log  N_r(E)}{-\log r} \ \text{ and }\ 
\ubd E\ =\ \limsup_{r\to 0} \frac{\log  N_r(E)}{-\log r} .
\end{equation} 
Note that there are various ways of defining $N_r$ that give equivalent values, so $N_r(E)\asymp N'_r(E)$, with the implied constants depending only on $d$,  leading to the same values of the dimensions. In particular, we can take $  N_r(E)$ to be the greatest number of points in an $r$-separated subset of $E$, or alternatively the number of cubes  in the lattice of half-open cubes of side length $r/\sqrt{d}$ (so diameter $r$) which intersect $E$, see \cite{FaBk}. 

In this section, we will review several ways of representing the box dimensions of a set $E$ in terms of measures supported by $E$. The following simple lemma is key in relating covering numbers to integrals.

\begin{lem}\label{boxcount}
Let $E\subset \mathbb{R}^d$ be bounded and non-empty. Then 
\begin{equation}\label{nre}
N_r(E)^{-1} \asymp \inf_{\mu \in  {\mathcal P}(E)} \int \mu(B(x,r))\mathrm{d}\mu(x) = \inf_{\mu \in  {\mathcal P}(E)}(\mu\times\mu)\{(x,y) \in E \times E :|x-y|\leq r\},
\end{equation} 
where the implied constants depend only on $d$ and the particular definition of $N_r$ in use.
\end{lem}
\begin{proof}
The infimum is attained by putting point masses of weight $1/N_r(E)$ on each of the points of an $r$-separated subset of $E$ containing $N_r(E)$ points. 

On the other hand, taking $N_r(E)$ to be the number of half-open cubes $\{B_i\}_{i=1}^{N_r(E)}$ in the lattice of cubes of diameter $r$ which intersect $E$, for all $\mu \in  {\mathcal P}(E)$,
$$1 =\bigg(\sum_{i=1}^{N_r(E)} \mu  (B_i)\bigg)^2 \leq \bigg(\sum_{i=1}^{N_r(E)} 1^2\bigg) \bigg(\sum_{i=1}^{N_r(E)}\mu  (B_i)^2\bigg) \leq N_r(E) \int \mu\big(B(x,r)\big)\mathrm{d}\mu(x)$$                                                                                                                                                                                                                                                                                                                                                                                                                                                                                                                                                                               
using the Cauchy-Schwarz inequality.
\end{proof}

\begin{prop}\label{boxprop}
Let $E \subset \mathbb{R}^d$ be bounded. Then
\begin{align}
\lbd E &=\ \liminf_{r \rightarrow 0}\frac{\log\Big(\inf_{\mu \in \mathcal{P}(E)}\int 
\mu(B(x,r))\mathrm{d}\mu(x)\Big)}{ \log r}\label{intdef1}\\
\mbox{and} \quad 
\ubd E &=\ \limsup_{r \rightarrow 0}\frac{\log\Big(\inf_{\mu \in \mathcal{P}(E)}\int 
\mu(B(x,r))\mathrm{d}\mu(x)\Big)}{ \log r}.\label{intdef}
\end{align}
\end{prop}
\begin{proof}
These identities follow immediately from \eqref{bdimdef} and \eqref{nre}. 
\end{proof}
It is not difficult to see that $N_r(E)^{-1} \asymp\inf_{\mu\in\mathcal{P}(E)}\sup_{x\in E}\mu(B(x,r))$. Consequently, we have 
\begin{align}
	\lbd E &=\ \liminf_{r \rightarrow 0}\sup_{\mu \in \mathcal{P}(E)}\inf_{x\in E}\frac{\log  \mu(B(x,r))}{ \log r}\label{lboxfrost}\\
	\mbox{and} \quad 
	\ubd E &=\ \limsup_{r \rightarrow 0}\sup_{\mu \in \mathcal{P}(E)}\inf_{x\in E}\frac{\log  \mu(B(x,r))}{ \log r}.\label{uboxfrost}
\end{align}

The drawback of \eqref{intdef1}-\eqref{uboxfrost} is that the measures $\mu\in \mathcal{P}(E)$ that minimise (or nearly minimise) the integrals  depend on the scale $r$. 
However, there is an anti-Frostman result that avoids this.

\begin{prop}\label{boxprop2}
Let $E \subset \mathbb{R}^d$ be Borel. Then
\begin{align}
\lbd E &=\min_{\mu\in \mathcal{P}(E)}\bigg\{ \liminf_{r\to 0} \sup_{x\in E}\frac{\log  \mu(B(x,r))}{ \log r} \bigg\} \label{lboxaf}\\
\mbox{and} \quad 
\ubd E &=\min_{\mu\in \mathcal{P}(E)}\bigg\{ \limsup_{r\to 0} \sup_{x\in E}\frac{\log  \mu(B(x,r))}{ \log r} \bigg\}. \label{uboxaf}
\end{align}
\end{prop}
\begin{proof}
This is essentially \cite[Theorem 2.1]{FFK}.  For \eqref{uboxaf}, take  $N_r(E)$ to be the largest number of disjoint balls of radii $r$  with centres in $E$, say $\{B(x_i,r)\}_{i=1}^{N_r(E)}$.  If  $\mu \in  {\mathcal P}(E)$ and $\mu(B(x,r))\geq r^s$ for all $x\in E$ and $0<r<1$,  then
	$1\geq \sum_{i=1}^{N_r(E)}\mu(B(x,r))\geq N_r(E)r^s,$
	giving  $\ubd E \leq s$.
	
	Now take $N_{2^{-k}}(E)$ to be the maximal number of points of $E$ in an $2^{-k}$-separated set, say $\{x_{k,i}\}_i$, which implies that for all $x\in E$ the ball $B(x, 2^{-k+1})$ contains at least one of the $x_{k,i}$. Define $\mu \in  {\mathcal P}(E)$ by
	$$\mu = c\sum_{k\in \mathbb{N}} \frac{1}{k^2} \sum_{i=1}^{N_{2^{-k}}(E)} \frac{1}{N_{2^{-k}}(E)}\delta_{x_{k,i}},$$
	where $\delta_{x_{k,i}}$ is a point mass at $x_{k,i}$ and $c$ is a normalising constant. 
	It is easy to see that $\mu$ is supported on $E$.
	Let $t>s>\ubd E$.
	It is easily seen that $\mu(B(x,2.2^{-k}))\geq\frac{1}{ k^{2}N_{2^{-k}}(E)}\geq 2^{k(s-t)}2^{-ks}$ for all $x\in E$ and $k $ large enough, 
	which gives $\mu(B(x,r))\geq r^t $ 
	for all sufficiently small $r$. In particular this measure attains the minimum in \eqref{uboxaf}.
	
	Identity \eqref{lboxaf} may be verified in a similar way by summing over $k$ in a suitable subset of $\mathbb{N}$.
\end{proof}

It is possible to express the box-counting definitions \eqref{bdimdef} in terms of energy integrals with respect to certain kernels, see \cite{Fa}. As well as giving alternative forms of \eqref{intdef1} and \eqref{intdef} this turns out to be useful in studying the box dimensions of projections and other images of sets, but we do not discuss this further here.

\subsection{Modified box dimensions}
Whilst box dimensions have some advantages, such as often being easy to calculate or estimate, mathematically they have several shortcomings. In particular,  lower box dimension is not finitely stable (that is the dimension of  union of two sets may differ from the maximum dimension of the individual sets), and neither lower nor upper box dimension is  stable for a countable family of sets. This can be remedied by `modifying' the definitions by considering countable covers $\{E_i\}_i$ of a set $E$ and taking the infimum  value of $\max_i \{\bdd E_i\}$ over all such covers for the modified box dimension of $E$.

For upper box dimension, this procedure leads to an alternative characterisation of packing dimension. On the other hand, modified lower box dimension has been little studied.

\subsubsection{Modified upper box dimension -- packing dimension}

We recall the definition of packing dimension via packing measures. A $\delta${\it-packing} of $E\subset \R^d$ is a finite or countable family of disjoint closed balls of radii at most $\delta$  with centres in $E$. For $s\geq 0$, the {\it $s$-dimensional packing premeasure} is defined as	
\begin{equation*}
	\mathcal{P}_0^s(E)=\inf_{\delta>0}{\mathcal P}_\delta^s(E),
\end{equation*}
	where
\begin{equation*}
		\mathcal{P}_\delta^s(E)=\sup\Big\{ \sum_i |B_i|^s:\{B_i\}_i \textit{ is  a $\delta$-packing of  $E$}\Big\}.
\end{equation*}
	The {\it $s$-dimensional packing measure} is then
\begin{equation*}
		\mathcal{P}^s(E)=\inf\bigg\{ \sum_{i=1}^\infty	\mathcal{P}_0^s(E_i): E\subset \bigcup_{i=1}^\infty E_i\bigg\}.
		\end{equation*}
Then $\mathcal{P}^s$ is a Borel measure.		
The {\it  packing dimension} of $E$ is defined analogously to Hausdorff dimension as
\begin{equation*}
	\pkd E=\inf\{  s\geq0: \mathcal{P}^s(E)=0\}=\sup\{ s\geq 0: \mathcal{P}^s(E)=\infty\}.
	\end{equation*}	
Packing dimension can be  characterised in terms of upper box dimension, that is as {\it modified upper box dimension} $\mubd$:
\begin{equation}\label{mpddef}
	\pkd E=\mubd E:= \inf\bigg\{{\sup_i }  \ \ubd E_i: E\subset \bigcup_{i=1}^\infty E_i, \text{ where the } E_i \text{ are compact} \bigg\},
\end{equation}
where the infimum is over all countable covers of $E$.
See, for example, \cite{FaBk,Tri}, for details of this equivalence. By virtue of the countable unions in the definition, it is immediate that $\mubd$ is countably stable. Moreover, in many applications of packing dimension, the form \eqref{mpddef} is used rather than the measure definition.

Cutler \cite{Cu} gives a detailed treatment of Frostman and anti-Frostman-type results for packing dimensions, see also \cite[Chapter 2]{Fa3}. In particular:

\begin{lem} {\rm \cite[Lemma 3.4]{Cu}}\label{frostman-packing}
	Let $E$ be a non-empty compact subset of $\R^d$. For every $0\leq s<\pkd E=\mubd E$, there exists $\mu\in\mathcal{P}(E)$  such that for all $x\in\R^d$ there are arbitrarily small  positive $r$ such that $\mu(B(x,r)) \leq r^{s}$. 	
\end{lem}
 Due to the analytic approximation lemma  (see \cite{Cu}, Lemma 2.5),  the result in Lemma \ref{frostman-packing} remains valid for all analytic sets.

\begin{lem} {\rm\cite{Cu,Fa3}}\label{re-packing-upper-local-dimension}
	Let $E\subset \R^d$ be a Borel set and let $\mu$ be a Borel measure with
	 $\mu(E)>0$. If for all $x\in E$
		\begin{equation*}
		\limsup_{r\to 0}\frac{\log\mu(B(x,r))}{\log r}\geq s.
	\end{equation*}
	then $\pkd E=\mubd E \geq s$. 
\end{lem}

	Let $E\subset \R^d$ be a Borel set.
From Lemmas \ref{frostman-packing} and \ref{re-packing-upper-local-dimension}, we deduce the following equivalent definitions for  packing or modified upper box dimension.
	\begin{align}\label{packingdimension}
	\pkd  E =&\sup_{\mu\in\mathcal{P}(E)}\inf_{x\in E}\limsup_{r\to 0}\frac{\log\mu(B(x,r))}{\log r}\notag\\	
	=&\sup_{\mu\in\mathcal{P}(E)}\inf_{x\in E}\limsup_{n\to \infty}\frac{\log\mu(C_n(x))}{-n\log 2}\notag\\
	=&\sup\{s\geq0:\text{there exists }\mu\in\mathcal{P}(E)\text{ and } c>0 \text{ such that for each }x\in\R^d, \text{there is }\notag\\
	&\qquad  \{r_n(x)\}_n\searrow0 \text{ such that } \mu(B(x,r_n(x)))\leq cr_n(x)^{s}\}\\	
	=&\sup\{s\geq0: \text{there exists }\mu\in\mathcal{P}(E)  \text{ such that for each }x\in\R^d, \mu(C_{n}(x))\leq  2^{-ns}\notag\\ 
	&\qquad 	\text{ for infinitely many } n\}.\notag
	\end{align}

\subsection{Modified lower box dimension}

Analogously to \eqref{mpddef} we write $\mlbd(E)$ for the {\it modified lower box dimension} of $E \subset  \mathbb{R}^d$, defined by 
$$\mlbd E:= \inf\bigg\{ {\sup_i }  \  \lbd E_i: E\subset \bigcup_{i=1}^\infty E_i, \text{ where the } E_i \text{ are compact} \bigg\}.$$
It follows directly from the definitions that 
$\hdd E\leq \mlbd E\leq \lbd E$.

Modified lower box  dimension seems to have been neglected compared to its upper counterpart and we now obtain a Frostman-type characterisation involving balls with radii of subsequences of given sequences. Parts of the proof follow the lines of  \cite[Lemma 3.4]{Cu}.

\begin{theo}\label{MB}
Let $E \subset  \mathbb{R}^d$ be compact. Then
\begin{eqnarray*}
\mlbd E 
	&= &\sup \{s  \geq 0:  \exists F\subset E  \text{{\rm  \ s.t.  for any sequence $\{r_k\}_k\searrow 0$ there exists $\{r_{k_i}\}_i\subset\{r_k\}_k$} and}\notag\\ 
&&  \text{{\rm  a Borel  measure $\mu$ with $ \spt\mu=F$ s.t.  $\mu(B(x,r_{k_i}))\leq r_{k_i}^{s }$  for all  $x \in \mathbb{R}^d$ and $i\in \mathbb{N}$\}}}\\
&=& \sup_{F\subset E}\inf_{\{r_k\}_k\searrow 0}  \sup_{\mu:\spt\mu=F}\limsup_{k\to\infty}\inf_{x\in E}\frac{\log\mu(B(x,r_k))}{\log r_k}.
\end{eqnarray*}
\end{theo}  
\begin{proof}
		We prove the first equality from which the second is immediate.
	Let $s$ be such that there exists $F\subset E$ such that for every sequence $\{r_k\}_k\searrow 0$ there exists a subsequence $\{r_{k_i}\}$ and  a Borel measure $\mu$ with $\spt\mu=F$  such that for  all   $i\in\mathbb{N}$,
	\begin{equation}\label{ineq}
	\sup_{x\in\R^d}	\mu(B(x,r_{k_i}))\leq  r_{k_i}^s.
	\end{equation}
	Let $\{E_j\}_j$ be an arbitrary countable  family  of closed subsets of  $E$ satisfying $E=\cup_j E_j$. 
	Then $F=\cup_j( E_j\cap F) $. 
	Since $ F$ is compact, by Baire's category theorem, there  exists an  index  $j_0$ such that $E_{j_0}\cap F$ has non-empty interior relative to $F$. 
	We claim that 
	\begin{equation}\label{claim}
		\lbd E_{j_0}\geq s.
	\end{equation}
	 Otherwise, there exists $\epsilon>0$ such that $\lbd E_{j_0}<s-\epsilon$.  By \eqref{lboxfrost}, there exists a decreasing sequence $\{r_k\}_k\searrow 0$ such that 
	\begin{equation}\label{ineq2}
		\sup_{\mu \in \mathcal{P}(E_{j_0})}\inf_{x\in E_{j_0}}\frac{\log  \mu(B(x,r_k))}{ \log r_k}<s-\epsilon.
	\end{equation}
	 For any Borel measure $\mu $ with $\spt\mu=F$,  the non-empty interior of $E_{j_0}\cap F$ implies  $\mu(E_{j_0})>0$.  Thus, $$\widetilde{\mu}=\dfrac{\mu|_{E_{j_0}}}{\mu(E_{j_0})} \in \mathcal{P}(E_{j_0}).$$
	According to \eqref{ineq2},  we have  $$\sup_{x\in E_{j_0}}\widetilde{\mu}(B(x,r_k))>r_k^{s-\epsilon}\quad\text{for all }k.$$
	Thus,  for sufficiently large $k$,
	\begin{equation*}
		\sup_{x\in E_{j_0}} {\mu}(B(x,r_k))>\mu(E_{j_0})\cdot r_k^{s-\epsilon}>r_k^s,
	\end{equation*}
	 which contradicts  \eqref{ineq}.  Therefore, we have proved that \eqref{claim} holds.  Consequently, $$\mlbd E \geq s.$$

	For the opposite inequality, assume that $E$ is   contained in the interior of the unit cube $C_0=[0,1]^d$. Let $0\leq s<\mlbd E$.		
	Define $\D_n$ as the collection of all dyadic half-open $n$th-level cubes of side lengths $2^{-n}$ that are contained in $C_0$.
	Set
	\begin{equation}\label{Dns}
		\D_n^s=\{C_n\in \D_n: \mlbd(E\cap C_n)>s\}.
	\end{equation}
   Since $\mlbd$ is finitely stable, for   $n\in\mathbb{N}$, 
   $$\mlbd E=\max_{C_n\in \D_n}\mlbd(E\cap C_n),$$  
   so  $\D_n^s\neq\emptyset$.
	Define   $$F_n=\bigcup_{C_n\in \D_n^s}C_n \quad \textit{ and }\quad F=\bigcap_n F_n.$$ 
	If $x\notin E$, then  $C_n(x)\cap F_n=\emptyset$ for $n$ sufficiently large    (where $C_n(x)$ is the $n$th-level dyadic cube containing $x$), which implies that  $x\notin F$. Then $	F\subset  E$.
	 It is clear that
	\begin{equation*}
		C_n\cap E=C_n\cap \big(F\cup(E\setminus F)\big)=\big(C_n\cap F\big)\cup\big(\bigcup_m\bigcup_{C_m\in\D_m\setminus\D_m^s}C_n\cap E\cap C_m\big).
	\end{equation*}
	 By   (\ref{Dns}), 
	 $$\mlbd\big(\bigcup_m\bigcup_{C_m\in\D_m\setminus\D_m^s}C_n\cap E\cap C_m\big)\leq s.$$ 
	 Then for $C_n\in \D_n^s$,
	\begin{equation}\label{ineq-s}
		\lbd\big(C_n\cap F\big)\geq \mlbd\big(C_n\cap F\big)>s,
	\end{equation}
a property that we use repeatedly.	 Let $\{r_k\}_k$ be an arbitrary   sequence decreasing to zero. 
For each $r_k$, there exists a unique integer $n_k$ such that
	\begin{equation}\label{rknk}
		2^{-(n_k+2)}\leq r_k< 2^{-(n_k+1)}.
	\end{equation}
We will define inductively an increasing sequence of integers $\{k_m\}_m$, set collections  $\mathcal{C}_{m} \subset  \D_{n_{k_m}}^s$, and probability measures $\mu_m$ 
supported by $\bigcup \{\overline{C}_{n_{k_m}}: C_{n_{k_m}} \in\mathcal{C}_{m}\}$
such that for all  $C_{n_{k_i}}\in\mathcal{C}_{i}$ with $1\leq i\leq m$,
\begin{equation}\label{indcon}
		\mu_m(C_{n_{k_i}})\leq  2^{-(d+2s)}\,2^{-n_{k_i}s}.
\end{equation}

To begin we choose $n_{k_1}$ to be the least  integer in the sequence $\{n_k\}_k$ such that  
	 \begin{equation}\label{snk1}
	 	2^{-{n_{k_1}}s}\,\#\D_{n_{k_1}}^s\geq 2^{d+2s},
	 \end{equation}  
	 where $\#$ denotes cardinality.
	We index the cubes in $\D_{n_{k_1}}^s$ and set  
	$$\mathcal{C}_{1}:=\D_{n_{k_1}}^s=\{C(j_1):1\leq j_1\leq \#\D_{n_{k_1}}^s\}.$$
	Let $\mu_1$  be the probability measure whose mass is uniformly distributed over each cube in $ \mathcal{C}_{1}$, that is
	\begin{equation}\label{mu1}
		\mu_1=\sum_{j_1}\frac{1}{\#\D_{n_{k_1}}^s}\dfrac{\mathcal{L}^d|_{C(j_1)}}{\mathcal{L}^d(C(j_1))},
	\end{equation}
	where $\mathcal{L}^d$ is $d$-dimensional Lebesgue measure and $\mathcal{L}^d|_C$ is its restriction to a set $C$.
	By \eqref{snk1} and \eqref{mu1}, for all $C_{n_{k_1}}\in \mathcal{C}_{1}$, 
	\begin{equation*}\label{k-1-1}
		\mu_1(C_{n_{k_1}})=\frac{1}{\#\D_{n_{k_1}}^s}\leq 2^{-(d+2s)}\,2^{-n_{k_1}s}.
	\end{equation*}
	By (\ref{ineq-s}), for all $C(j_1)\in\mathcal{C}_1$, 
	$$\underline{\dim}_B\big(F\cap C(j_1)\big)\geq\underline{\dim}_{MB}\big(F\cap C(j_1)\big)> s.$$
	
	Assume inductively that we have selected  integers $\{n_{k_i}\}_{1\leq i\leq m-1}\subset \{n_k\}_k$ and  set collections $\{\mathcal{C}_{i}\}_{1\leq i\leq m-1}$, where the cubes in $\mathcal{C}_{i}$ are labeled as follows:
	\begin{align*}
		\mathcal{C}_{i}=\big\{C(j_1,\cdots,j_i):1\leq j_1\leq \#\D_{n_{k_1}}^s,1\leq j_2\leq\#\D_{n_{k_2}}^s(j_1),\cdots, 1\leq j_i\leq\#\D_{n_{k_i}}^s(j_1,\cdots,j_{i-1})\big\},
	\end{align*}
	where 
	$$\D_{n_{k_l}}^s(j_1,\cdots,j_{l-1})=\{C_{n_{k_l}}\in \D_{n_{k_l}}^s:
	C_{n_{k_l}} \subset C(j_1,\cdots,j_{l-1})\}.$$
	Assume further that  for $1\leq i\leq m-1$, probability measures $ \mu_i $ are  defined such that 	for  $C_{n_{k_l}}\in\mathcal{C}_{l}$ with $1\leq l\leq i\leq m-1$,
	\begin{equation}\label{li}
		\mu_i(C_{n_{k_l}})\leq 2^{-(d+2s)}\,2^{-n_{k_l}s}.
	\end{equation}
	Using \eqref{ineq-s} we select $k_m>k_{m-1}$ such that $n_{k_m}$ is the least integer in the sequence $\{n_k\}_k$ such that  for all multi-indices  $(j_1,\cdots,j_{m-1})$,
	\begin{equation}\label{nkmineq}
			2^{-{n_{k_m}}s}\,\#\D_{n_{k_m}}^s(j_1,\cdots,j_{m-1})\, \geq 2^{d+2s}\mu_{m-1}\big(C(j_1,\cdots,j_{m-1})\big).
	\end{equation}
		Index the cubes in each  $\D_{n_{k_m}}^s(j_1,\cdots, j_{m-1})$ as
		 $$\big\{C(j_1,\cdots, j_{m-1},j_m):1\leq j_m\leq \#\D_{n_{k_m}}^s(j_1,\cdots, j_{m-1})\big\}$$
	and set
	\begin{equation*}
		\mathcal{C}_{m}=\big\{C(j_1,\cdots, j_{m-1},j_m):1\leq j_1\leq \#\D_{n_{k_1}}^s,\cdots, 1\leq j_m\leq \#\D_{n_{k_m}}^s(j_1,\cdots, j_{m-1})\big\}.
	\end{equation*}	
	We now define $\mu_m$ 	as the measure such that for each multi-index  $(j_1,\cdots,j_{m-1})$, the measure $\mu_{m-1}$ of  $C(j_1,\cdots, j_{m-1})$  is uniformly distributed over all the $n_{k_m}$th-level cubes in $\D_{n_{k_m}}^s(j_1,\cdots, j_{m-1})$, that is,
	\begin{equation}\label{mum}
		\mu_m=\sum_{(j_1,\cdots,j_{m-1})}\sum_{j_m}\frac{\mu_{m-1}\big(C(j_1,\cdots, j_{m-1})\big)}{\#\D_{n_{k_m}}^s(j_1,\cdots, j_{m-1})}\, \dfrac{\mathcal{L}^d|_{C(j_1,\cdots,j_m)}}{\mathcal{L}^d\big(C(j_1,\cdots,j_m)\big)}.
	\end{equation}
	By \eqref{nkmineq}, for   each $C_{n_{k_m}}\in\mathcal{C}_m$,
	\begin{equation*}
		\mu_m(C_{n_{k_m}})=\frac{\mu_{m-1}\big(C(j_1,\cdots, j_{m-1})\big)}{\#\D_{n_{k_m}}^s(j_1,\cdots, j_{m-1})}\leq 2^{-(d+2s)}\, 2^{-n_{k_m}s}.
	\end{equation*}
	By \eqref{li},  it follows  that for $C_{n_{k_i}}\in\mathcal{C}_i$ with  $1\leq i<m$,
	\begin{equation*}
		\mu_m(C_{n_{k_i}})=\mu_{m-1}(C_{n_{k_i}})\leq 2^{-(d+2s)}2^{-n_{k_i}s}.
	\end{equation*}
This completes the inductive step and establishes \eqref{indcon}.	 
	 
	Let $\{\mu_{m_l}\}_l$ be a weakly convergent subsequence of $\{\mu_m\}_m$ with limit $\mu$, say. 
	Define $G_n=\R^d\setminus \overline{F_n}$ and $G=\R^d \setminus \overline{F}=\bigcup_n G_n$.
	Then $\mu_{m} $ is   fully  supported on $\overline{F}_{n_{k_m}}$ and $G_n$ is increasing, so 
	 for  $m\geq i$,  
	 	$$\mu_m(G_{n_{k_i}})\leq \mu_m(G_{n_{k_m}})=0.$$
	Since $G_{n_{k_i}}$ is open, 
	$$\mu(G_{n_{k_i}})\leq\liminf_m\mu_{m}(G_{n_{k_i}})=0.$$
	Consequently, $$\mu(G)=\mu\big(\bigcup_iG_{n_{k_i}}\big)=0,$$ which implies that $\mu$ is  supported on $\overline{F}$.
 For any $x\in \overline{F}$ and any open set $V$ containing $x$, there exists $C_{n_{k_m}}\in\mathcal{C}_m$ such that $x\in\overline{C}_{n_{k_m}}$ and $\overline{C}_{n_{k_m}}\subset V$. It follows that
		\begin{equation*}
			\mu(V)\geq\mu(C_{n_{k_m}})=\mu_m(C_{n_{k_m}})>0.
		\end{equation*}
		That is, $\spt\mu= \overline{F}$.
		Since $E$ is compact and $F\subset  E$, it is clear that $\overline{F}\subset  E$.
	By \eqref{rknk}, for every $x\in\R^d$ and $i\in\mathbb{N}$, the ball $B(x,r_{k_i})$ is contained in a interior of the union of $2^d$ cubes in $\D_{n_{k_i}}$ which we denote by $U(x,r_{k_i})$.
Then
\begin{align*}
	\mu(B(x,r_{k_i}))&\leq\mu\big(U(x,r_{k_i})\big)\\
	&\leq\liminf_{m}\mu_{ m}\big(U(x,r_{k_i})\big)\\
	&\leq 2^d \liminf_{m}\max_{C_{n_{k_i}}\in \D_{n_{k_i}}}\mu_{m}\big(C_{n_{k_i}}\big)\\
	&\leq 2^d\cdot 2^{-(d+2s)}\cdot 2^{-n_{k_i}s}\\
	&\leq   r_{k_i}^s,
\end{align*}
as required.
\end{proof}

By minor modification, Theorem \ref{MB} may be expressed in a dyadic form:
\begin{eqnarray}\label{mlbdim1}
	\mlbd E &= &\sup \{s  \geq 0: \exists F\subset E \text{{\rm \ s.t.  for $\forall \{n_k\}_k\nearrow \infty, \exists\ \{n_{k_i}\}_i\subset\{n_k\}_k$ and a  Borel measure   }}\notag\\ 
	&& \mu\   \text{{\rm  with  $\spt\mu=F$ such that $\mu(C_{n_{k_i}})\leq 2^{-n_{k_i}s}$  for all  $C_{n_{k_i}}\in \D_{n_{k_i}}$\}}}\\
	&=& \sup_{F\subset E}\inf_{\{n_k\}\nearrow\infty} \sup_{\mu:\spt\mu=F} \limsup_{k\to\infty} \inf_{C_{n_k}\in \D_{n_k}} \frac{\log  \mu(C_{n_k})}{ -{n_k}\log 2}. \notag
\end{eqnarray}
In the same way as in \eqref{infsum} and Proposition \ref {cordim} the characterisations of  $\mlbd$ in Theorem \ref{MB} have an integral version:
	\begin{equation}\label{MBL}
		\mlbd E = \sup_{F\subset E}\inf_{\{r_k\}\searrow0}\sup_{\mu:\spt\mu=F}\limsup_{k\to\infty}   \frac{\log  \int\mu(B(x,r_k))\mathrm{d}\mu(x)}{ \log r_k}.	
	\end{equation}

\section{Relationships between dimensions}\label{sec3}
The main aim of this section is to demonstrate  that  modified lower box dimension, upper correlation dimension and packing dimension are all distinct, disproving a conjecture in \cite{Fr}.
\setcounter{equation}{0}
\setcounter{theo}{0}
\subsection{Inequalities}
Apart from the inclusion of  $\ucod E$ in the diagram below, the other  relationships are well-known and there are examples that show that sets $E$ exist with all possible values of these dimensions attained simultaneously, see for example, \cite{MO, Spe}. Here  $\dim_{\rm A }$ denotes Assouad dimension, and further relationships between Assouad type dimensions and lower dimensions may be found in \cite{FrBk}.

$$
\begin{array}{clccll}
    &\Nearrow  &  \lbd E &  \Searrow &\\
\hdd E\equiv\lcod E  \leq \mlbd E \!\!\! &  & && \!\!\! \ubd E  \leq \dim_{\rm A }E. \\
    &\Searrow \!\! & \ucod E \leq \pkd E \equiv \mubd E \!\!& \Nearrow &\\  
\end{array}
$$
Note that $\ucod E$ and $ \lbd E$ are not comparable.  Any countable set with positive box dimension will give an example such that $\ucod E< \lbd E$. On the other hand, Proposition \ref{example1} below exhibits a compact set such that  $ \ucod E>  \underline{\dim}_{MB} E  = \lbd  E$. Similarly, $\lbd E$ and $\pkd E$ are not comparable, see for example \cite{MO,Spe}.

The following proposition allows us to include upper correlation dimension in this picture.
\begin{prop}\label{ine}
For all Borel sets $E \subset \mathbb{R}^d$, 
\begin{equation}\label{ineqs}
\mlbd E \leq \ucod E \leq \pkd E,
 \end{equation}
\end{prop}
\begin{proof}
 According to Theorem \ref{MB},  it is clear that
		\begin{equation*}
			\mlbd E\leq \inf_{\{r_k\}_k\searrow 0}  \sup_{\mu\in\mathscr{P}(E)}\limsup_{k\to\infty}\inf_{x\in E}\frac{\log\mu(B(x,r_k))}{\log r_k}\leq\ucod E.
		\end{equation*}
		The second inequality in \eqref{ineqs} follows directly from \eqref{cd3} and \eqref{packingdimension}.
\end{proof}

We next present examples to show that strict inequality is possible in the inequalities \eqref{ineqs} in a wide sense; in particular the three dimensions in \eqref{ineqs} are essentially different. In what follows, by slight abuse of notation, given a family of sets $\mathcal{D}$ and a set $C$, we write $\mathcal{D} \cap  C $ for the sets in $\mathcal{D}$ that are subsets of $C$.

\begin{prop}\label{example1}
Given $0<t<s<1$, there exists a compact set $E\subset \bbbr^n$ such that $\mlbd E=\lbd E  = t$  and $\ucod E=\pkd E =s$.
\end{prop}

	\begin{proof}
		Fix positive real numbers $0<t<s<1$. Let $\{\epsilon_k\}_k\subset (0,t)$ be a sequence of real numbers  decreasing to $0$. Let  $n_0=0$ and $n_1=1$.  We define a sequence of integers $\{n_m\}_m$ recursively.
		 For $k\geq1$, let
		\begin{equation}\label{n2k}
			n_{2k}=\bigg\lfloor\frac{1}{t-\epsilon_k}\sum_{i=0}^{k-1}(n_{2i+1}-n_{2i})  \bigg\rfloor+1
		\end{equation}	
		and 
		\begin{equation}\label{n2k1}
			n_{2k+1}=\bigg\lfloor\frac{1}{1-s}\Big(n_{2k}-\sum_{i=0}^{k-1}(n_{2i+1}-n_{2i})\Big)\bigg\rfloor,
		\end{equation}	
	 where $\lfloor c\rfloor$ is  the integral part of the real number $c$.
		 Define $\mathcal{D}_n$ as the collection of all half-open dyadic intervals of length  $2^{-n}$  that are contained in $[0,1]$, which we refer to as $n$th-level intervals, so 
		\begin{equation*}
			\mathcal{D}_n=\bigg\{\Big(\frac{k}{2^n},\frac{k+1}{2^n}\Big]:k=0,1,\cdots,2^n-1\bigg\}.
		\end{equation*}

Beginning with $\mathcal{C}_0=[0,1]$, for all   $n\in\mathbb{N}$ we will select dyadic intervals from $\mathcal{D}_n$, and  denote the set of all such chosen intervals as $\mathcal{C}_n$. 
		For each $n$th-level interval chosen for $\mathcal{C}_n$,  we ensure that at least one of  its $(n+1)$th-level subintervals will be chosen for $\mathcal{C}_{n+1}$. 
		 Thus at the $(n+1)$th stage, we either choose the lefthand $(n+1)$th-level subinterval of  each interval in $\mathcal{C}_n$ (referred to as {\it Method 1})  or we choose both	$(n+1)$th-level subintervals of  each  interval in $\mathcal{C}_n$ (referred to as  {\it Method 2}).
		 
		At the first stage ($n_1=1$), we adopt Method 2. Thus
		\begin{equation*}
			\textstyle{\mathcal{C}_{n_1}=\big\{(0,\frac{1}{2}],(\frac{1}{2},1]\big\}.}
		\end{equation*}
		We alternate between using Method 1 and Method 2 throughout the construction process.
		At the $m$th stage, we use Method 1 when $n_{2k-1}<m\leq n_{2k}$ for an integer $k$,  and Method 2 when $n_{2k}<m\leq n_{2k+1}$.
			Define
		\begin{equation*}
			E_n=\bigcup_{C_n\in \mathcal{C}_n}C_n \quad\textit{and}\quad E=\bigcap_{n =1}^\infty \overline{E_n}\, ,
		\end{equation*}
		where $\overline{E_n}$ is the closure of $E_n$.
		
		Clearly, when $n_{2k-1}\leq m\leq n_{2k}$,
		\begin{equation*}
			\#(\mathcal{C}_{n_{2k}}\cap C_{m})=\left\{
			\begin{aligned}
				&1 \quad&\text{if } C_{m}\in\mathcal{C}_{m};\\
				&0 \quad&\text{if } C_{m}\notin\mathcal{C}_{m}.
			\end{aligned}\right.
		\end{equation*}
		Thus 
		\begin{equation*}
			\#\mathcal{C}_{n_{2k}}=\prod_{i=0}^{k-1}2^{n_{2i+1}-n_{2i}}.
		\end{equation*}
		By \eqref{n2k}, 
		\begin{equation}\label{2k}
			2^{n_{2k}(t-\epsilon_k)-1}\leq\#\mathcal{C}_{n_{2k}}<2^{n_{2k}(t-\epsilon_k)}
		\end{equation}
		When $n_{2k}\leq m\leq n_{2k+1}$,
		\begin{equation*}
			\#(\mathcal{C}_{n_{2k+1}}\cap C_{m})=\left\{ 
			\begin{aligned}
				&2^{n_{2k+1}-m} \quad&\text{if } C_{m}\in\mathcal{C}_{m};\\
				&0 \quad&\text{if } C_{m}\notin\mathcal{C}_{m}.
			\end{aligned}\right.
		\end{equation*}
		Consequently,
		\begin{equation*}
			\#\mathcal{C}_{n_{2k+1}}=\prod_{i=0}^{k}2^{n_{2i+1}-n_{2i}}.
		\end{equation*}
		It follows from \eqref{n2k1} that
		\begin{equation}\label{2k+1}
			2^{n_{2k+1}s-1}<\#\mathcal{C}_{n_{2k+1}}\leq2^{n_{2k+1}s}.
		\end{equation}
		
		It follows directly from \eqref{2k} that $\lbd E\leq t$.
		We next show that $\mlbd E\geq t$.
		Let $\epsilon>0$ and
		let $E=\bigcup_i E_i$ be an arbitrary countable cover of $E$ with each  $E_i$  closed.
		Since $E$ is compact, by Baire's category theorem there  exists an  $E_i$ such that $E \cap E_i$ has non-empty interior relative to $E$.
		Thus, for some $n$, there exists a $n$th-level interval $C_n$ such that $E\cap C_n\subset  E_i$.
		Since $E\cap C_n\neq\emptyset$, it follows that for $m>n$,
		\begin{align*}
			N_{2^{-m}}(E_i)&\geq N_{2^{-m}}(E\cap C_n)\\
			&\geq\frac{\#\mathcal{C}_m}{\#\mathcal{C}_n}\\
			&=\frac{1}{\#\mathcal{C}_n}\begin{cases}
				\#\mathcal{C}_{n_{2k}}, & \text{if } n_{2k-1}<m\leq n_{2k}\\
					\#\mathcal{C}_{n_{2k}}\cdot 2^{m-n_{2k}}, & \text{if } n_{2k}<m\leq n_{2k+1}
			\end{cases}\\
			&\geq\frac{1}{\#\mathcal{C}_n}\begin{cases}
				2^{n_{2k}(t-\epsilon_k)-1}, & \text{if } n_{2k-1}<m\leq n_{2k}\\
				2^{n_{2k}(t-\epsilon_k)-1+(m-n_{2k})}, & \text{if } n_{2k}<m\leq n_{2k+1}
			\end{cases}\quad\big(\text{by } \eqref{2k}\big)\\
			&\geq\frac{1}{2\#\mathcal{C}_n}2^{m(t-\epsilon)},
		\end{align*}
		where the last inequality is valid if $m$ is large enough so that the corresponding $k$ satisfies $\epsilon_k<\epsilon$.
		Thus, $\lbd E_i\geq t-\epsilon$. According to the definition of the modified lower box dimension and the arbitrariness of $\epsilon$, we conclude $\mlbd E\geq t$
		and thus
		\begin{equation*}
			\mlbd E=\lbd E= t.
		\end{equation*}

		We next  verify that  $\ucod (E)= \pkd E = s$. For each $m\in \mathbb{N}$, define a measure 
		\begin{equation*}
			\mu_m= \sum_{C\in\mathcal{C}_{n_{2m+1}}}\frac{1}{\#\mathcal{C}_{n_{2m+1}}}\frac{\mathcal{L}|_C}{\mathcal{L}(C)},
		\end{equation*}
		where $\mathcal{L}$ is Lebesgue measure.
		For  $C_{n_{2k+1}}\in\mathcal{C}_{n_{2k+1}}$ and $k\leq m$, 
			\begin{align}\label{mumc}
			\mu_m(C_{n_{2k+1}})
			=&\ \#(\mathcal{C}_{n_{2m+1}}\cap C_{n_{2k+1}})\cdot\frac{1}{\#\mathcal{C}_{n_{2m+1}}}\notag\\
			=&\ \frac{1}{\#\mathcal{C}_{n_{2k+1}}}\notag\\
			\leq&\ 2^{-n_{2k+1}s+1}\quad\big(\text{by } \eqref{2k+1}\big).
		\end{align}	
		
		Let $\{\mu_{m_j}\}_j$ be a weakly convergent subsequence of $\{\mu_m\}_m$, and let $\mu$ be the limit of this subsequence.
		It is easy to see that $\mu\in\mathcal{P}(E)$.
		For each $x\in E$, $x$ lies in the interior of the union of two ${n_{2k+1}}$th adjacent dyadic intervals, denoted as $U_k(x)$.
			Thus
		\begin{align*}
			\mu(\{x\})\leq&\ \mu(U_k(x))\\
		\leq&\ \liminf_m\mu_m(U_k(x))\\
		\leq&\  2^{-n_{2k+1}s+2}\quad\big(\text{by } \eqref{mumc}\big).
		\end{align*}	
			Letting  $k$ tend to infinity  implies that $\mu$ has no atoms.
			We denote the interior of  ${C}_{n_{2k+1}}$ by ${\rm int }\,{C}_{n_{2k+1}}$ and its boundary by  $\partial C_{n_{2k+1}}$. Since $\mu(\partial C_{n_{2k+1}})=0$, 
		\begin{align*}
			\mu(C_{n_{2k+1}})\leq&\ \mu(\partial C_{n_{2k+1}})+\mu({\rm int }\,{C}_{n_{2k+1}})\\
			\leq&\ \liminf_m\mu_{m}({\rm int }\,{C}_{n_{2k+1}})\\
			\leq&\ 2^{-n_{2k+1}s+1}\quad\big(\text{by } \eqref{mumc}\big).
		\end{align*}	
By Proposition \ref{cordim},  $\ucod E\geq s$.
		
		If $\pkd  E> s$, then there exists $\epsilon_0>0$ such that $\ubd(E)\geq\pkd E> s+\epsilon_0$.
		Then  there exists a sequence of integers $\{m_k\}_k$ such that
		\begin{equation}\label{contra}
			\#\mathcal{C}_{m_k}\geq 2^{m_k(s+\epsilon_0)}.
		\end{equation}
		However, for each $m\in\mathbb{N}$, 
			\begin{align*}
			\#\mathcal{C}_m&=\begin{cases}
				\#\mathcal{C}_{n_{2k-1}}, & \text{if } n_{2k-1}<m\leq n_{2k}\\
				\#\mathcal{C}_{n_{2k+1}}\cdot 2^{m-n_{2k+1}}, & \text{if } n_{2k}<m\leq n_{2k+1}
			\end{cases}\\
			&\leq\begin{cases}
				2^{n_{2k-1}s}, & \text{if } n_{2k-1}<m\leq n_{2k}\\
				2^{n_{2k+1}s+m-n_{2k+1}}, & \text{if } n_{2k}<m\leq n_{2k+1}
			\end{cases}\quad\big(\text{by } \eqref{2k+1}\big)\\
			&\leq 2^{ms},
		\end{align*}
		which contradicts  \eqref{contra}.
		Thus $$\ucod E= \pkd E =s.$$
	\end{proof}

The next example distinguishes  upper correlation dimension and packing dimension.

\begin{prop}\label{propeq}
Given $s,t$ with $0<t<s<1$, there exists a compact set $E$ such that $\ucod E=t$ and  $ \pkd E=s$.
\end{prop}
		\begin{proof}
		Set $N_0=1$,  and  for $k\geq1$, define 
		\begin{equation*}\label{nk}
			n_k=\max\Big\{\Big\lfloor\frac{N_{k-1}}{1-s}\Big\rfloor,\ \Big\lceil\frac{t}{s-t}\Big\rceil\Big\}  \text{ and } N_k=\max\Big\{\Big\lfloor\frac{sn_k}{t}\Big\rfloor,\  \Big\lceil\frac{1-s}{s}\Big\rceil\Big\},
		\end{equation*} 
		so $n_k < N_k <n_{k+1}$. Then for all $k\geq2$,
		\begin{equation}\label{nka}
		s n_k- (1-s) < n_k-N_{k-1} \leq s n_k
		\end{equation} 
		and
		\begin{equation}\label{nkb}
			sn_k-t<N_kt\leq sn_k.
		\end{equation}
		
		Let $C_0=[0,1]$. 
		We  let $\mathcal{D}_n$ denote the family of  half-open $n$th-level dyadic intervals of length $2^{-n}$ that are contained in $C_0$. 
		We select   $n$th-level dyadic intervals from  $\mathcal{D}_n$ in a specific way and denote the set of chosen intervals by $\mathcal{C}_n$. We also {\it designate} certain intervals, which we denote by a tilde.
		For $n=1$, we select all   intervals in $\mathcal{D}_1$, i.e., 
		$$\mathcal{C}_1=\mathcal{D}_1=\big\{(0,\textstyle{\frac{1}{2}}],(\textstyle{\frac{1}{2}},1]\big\}.$$ 
		Designate the interval  $\widetilde{C}_{1}=\big(\mathcal{C}_1\big)_L=(0,\frac{1}{2}]$, where  $\big(\mathcal{C}\big)_L$ denotes
		the  left-most interval in the   interval collection $\mathcal{C}$.
		 For $1<n\leq n_1$,  select all  $n$th intervals from  $\mathcal{D}_n\cap \widetilde{C}_{1}$.
		Simultaneously, from   $\mathcal{D}_n\cap (\frac{1}{2},1]$, we select only the  left-most interval $\big(\mathcal{D}_n\cap (\frac{1}{2},1]\big)_L$. 
		Thus $$\mathcal{C}_{n_1} =\big(\mathcal{D}_{n_1}\cap \widetilde{C}_{1}\big)\cup\big(\mathcal{D}_{n_1}\cap (\textstyle{\frac{1}{2}},1]\big)_L;$$ 
		For $n_1<n\leq N_1$ and  each $C_{n_1}\in\mathcal{C}_{n_1}$,  from $\mathcal{D}_n\cap C_{n_1}$  we select only the  left-most interval $\big(\mathcal{D}_n\cap C_{n_1}\big)_L$.
		Then
		\begin{equation*}
			\mathcal{C}_{N_1}=\Big\{\big(\mathcal{D}_{N_1}\cap {C}_{n_1}\big)_L:{C}_{n_1}\in\mathcal{C}_{n_1}\Big\}.
		\end{equation*}
		Designate  the interval  $\widetilde{C}_{N_1}=\big(\mathcal{C}_{N_1}\big)_L^{\widetilde{C}_{1}}=\big(\mathcal{D}_{N_1}\cap (\frac{1}{2},1]\big)_L$,
		where $\big(\mathcal{C}\big)_L^C$ denotes the  left-most interval  in the interval collection $\mathcal{C}$  to the right of the interval $C$.
		For $N_1<n\leq n_2$, we select all  $n$th-level intervals from $\mathcal{D}_n\cap \widetilde{C}_{N_1}$, and for each $C_{N_1}\in\mathcal{C}_{N_1}\setminus\{\widetilde{C}_{N_1}\}$, we only  select $\big(\mathcal{D}_n\cap {C}_{N_1}\big)_L$. Then we define  
		\begin{equation*}
			\mathcal{C}_{n_2}=\big(\mathcal{D}_{n_2}\cap \widetilde{C}_{N_1}\big)\cup\Big\{\big(\mathcal{D}_{n_2}\cap {C}_{N_1}\big)_L:{C}_{N_1}\in\mathcal{C}_{N_1}\setminus\{\widetilde{C}_{N_1}\}\Big\};
		\end{equation*}
		For $n_2<n\leq N_2$  and    each $C_{n_2}\in \mathcal{C}_{n_2}$, we  select only  $\big(\mathcal{D}_n\cap {C}_{n_2}\big)_L$.
	Consequently,
		\begin{equation*}
			\mathcal{C}_{N_2}=\Big\{\big(\mathcal{D}_{N_2}\cap {C}_{n_2}\big)_L:{C}_{n_2}\in\mathcal{C}_{n_2}\Big\}.
		\end{equation*}
		
For $m\geq2$, if $\widetilde{C}_{N_{m-1}}$ is the right-most interval in $\mathcal{C}_{N_{m-1}}$ we define $\widetilde{C}_{N_{m}}=\big(\mathcal{C}_{N_{m}}\big)_L$, otherwise    
$\widetilde{C}_{N_{m}}=\big(\mathcal{C}_{N_{m}}\big)_L^{\widetilde{C}_{N_{m-1}}}$	.	
For $N_{m}<n\leq n_{m+1}$, we form $\mathcal{C}_{n}$ by selecting all  $n$th-level intervals from $\mathcal{D}_n\cap \widetilde{C}_{N_m}$, but for each $C_{N_m}\in\mathcal{C}_{N_m}\setminus  \{\widetilde{C}_{N_m}\} $ we only  select $\Big(\mathcal{D}_n\cap {C}_{N_m}\Big)_L$. 
		That is,
		\begin{equation*}
			\mathcal{C}_{n}=\Big(\mathcal{D}_{n}\cap \widetilde{C}_{N_m}\Big)\cup\Big\{\Big(\mathcal{D}_{n}\cap {C}_{N_m}\Big)_L:{C}_{N_m}\in\mathcal{C}_{N_m}\setminus \{\widetilde{C}_{N_m}\}\Big\}.
		\end{equation*}	
			For $n_{m+1}<n\leq N_{m+1}$, we  select only $\big(\mathcal{D}_n\cap {C}_{n_{m+1}}\big)_L$ for each $C_{n_{m+1}}\in \mathcal{C}_{n_{m+1}}$.
		That is,
		\begin{equation*}
			\mathcal{C}_{n}=\Big\{\big(\mathcal{D}_{n}\cap {C}_{n_{m+1}}\big)_L:{C}_{n_{m+1}}\in\mathcal{C}_{n_{m+1}}\Big\}.
		\end{equation*}	
		By iterating this process, we generate a sequence of  families of sets $\{\mathcal{C}_{n}\}_n$. Define $$E=\bigcap_{n}\bigcup_{C_n\in\mathcal{C}_{n}} \overline{C_n}.$$
		
		It is clear that $\#\mathcal{C}_{N_m}=\#\mathcal{C}_{n_m}$ and $\#\mathcal{C}_{n_m}-\#\mathcal{C}_{N_{m-1}}=2^{{n_m}-N_{m-1}}-1$. It follows that for $k\geq2$,
		\begin{align}
			\#\mathcal{C}_{n_k}=&\sum_{m=1}^k\big(\#\mathcal{C}_{n_m}-\#\mathcal{C}_{N_{m-1}}\big)+\#\mathcal{C}_{1}\notag\\
			=&\sum_{m=1}^k\big(2^{{n_m}-N_{m-1}}-1\big)+2\notag\\
			\leq&k2^{n_ks}\quad\text{\big(by \eqref{nka}\big)}.\label{cnk}
		\end{align}
		For each $k\in\mathbb{N}$, define
		\begin{equation*}
			\mathcal{J}_k=\big\{j:\widetilde{C}_{N_j}\subset  \widetilde{C}_{N_k}\big\}.
		\end{equation*}
		By virtue of the construction of these sets,  $\mathcal{J}_k$   is formed of  countably many  disjoint segments of consecutive integers, which we can  write as 
		\begin{equation*}
			\mathcal{J}_k=\{{j_1},j_1+1,\cdots,j_1+l_1, j_2,j_2+1,\cdots,j_2+l_2,\cdots\},
		\end{equation*}
		where $j_i$ and $j_i+l_i$ are the end points of these segments. For $k\geq2$, we see that $j_{i+1}-(j_i+l_i)>2$.
		If $N_j\leq m\leq n_{j+1}$ with $j\in\mathcal{J}_k$, then
		\begin{align}\label{m1}
		\#\big(\mathcal{C}_m\cap \widetilde{C}_{N_k}\big)&\leq\#\mathcal{C}_{N_j}+2^{m-N_j}\notag\\
		&\leq j2^{n_js}+2^{m-N_j}\quad (\text{by \eqref{cnk}})\notag\\
		&\leq j2^{ms}+2^{ms}\quad (\text{by \eqref{nka}})\notag\\
		&\leq m 2^{ms}.
		\end{align}
		If $n_{j+1}\leq m\leq N_{j+1}$, then 
		\begin{align}\label{m2}
			\#\big(\mathcal{C}_{m}\cap \widetilde{C}_{N_k}\big)
			\leq\#\mathcal{C}_{n_{j+1}}\leq(j+1)2^{n_{j+1}s}
				\leq m2^{ms}.
		\end{align}
		 In particular, when $N_{j_i+l_i+1}\leq m\leq N_{j_{i+1}}$, 
		\begin{align}\label{less}
			\#(\mathcal{C}_m\cap\widetilde{C}_{N_k} )\leq&		\#\mathcal{C}_{n_{j_i+l_i+1}}\notag\\
		\leq&(j_i+l_i+1)2^{n_{j_i+l_i+1}s}\quad\text{\big(by \eqref{cnk}\big)}\notag\\
		\leq& (j_i+l_i+1)2^{N_{j_i+l_i+1}t+1}\quad\text{\big(by \eqref{nkb}\big)}\notag\\
		\leq& m2^{mt+1}.
		\end{align}
		We call $[N_{j_i+l_i+1},N_{j_{i+1}}]$ a {\it lower-count range}  of $\widetilde{C}_{N_k}$.
		
		We can now show by contradiction that $\ucod E\leq t$.  For if   $\ucod E > t+\epsilon_0$ for some $\epsilon_0>0$,   \eqref{cd4} implies that there exist a measure $\mu\in\mathcal{P}(E)$ and a sequence $\{m_l\}_l$  such that for every  designated interval  $\widetilde{C}_{N_k}$ with positive measure,  
		\begin{equation}\label{greater}
			\mu( \widetilde{C}_{N_k})\leq\#(\mathcal{C}_{m_l}\cap \widetilde{C}_{N_k})\cdot2^{-m_l(t+\epsilon_0)}.
		\end{equation}
We note that if $m$ is not in the lower-count range of $\widetilde{C}_{N_k}$ and $m>N_k$, then $N_{j_i}<m<N_{j_i+l_i+1}$ for some $j_i\in\mathcal{J}_k$.
	Suppose that $\widetilde{C}_{N_{k'}}\cap \widetilde{C}_{N_{k}}=\emptyset$ and $m>N_{k'}$. As $\mathcal{J}_{k'}\cap \mathcal{J}_k=\emptyset$,  it is easy to check that $m$ lies in a lower-count range of $\widetilde{C}_{N_{k'}}$.  Since $\mu$ is not a Dirac measure, we can find two   disjoint designated intervals $ \widetilde{C}_{N_{k}}$ and $\widetilde{C}_{N_{k'}}$  with positive measure. 
	If $m_l$ lies in the  lower-count range of $ \widetilde{C}_{N_{k}}$  for at most finitely many $l$,  then there are infinitely many $l$ such that $m_l$  are in the lower-count range of $ \widetilde{C}_{N_{k'}}$. 
	Combining \eqref{less} and \eqref{greater}, there are infinitely many $m_l$ such that 
	$$\mu( \widetilde{C}_{N_{k'}})\leq   m_l2^{m_lt+1}\cdot2^{-m_l(t+\epsilon_0)} = m_l2^{-m_l\epsilon_0+1},$$
	contradicting that $\mu(\widetilde{C}_{N_{k'}})>0$. If otherwise, $m_l$ lies in the  lower-count range of $ \widetilde{C}_{N_{k}}$  for infinitely many $l$ which also leads to a contradiction.

		 To show that $\pkd E=s$, first consider a partition $E=\bigcup_i E_i$ with each $E_i$  closed. Since $E$ is compact, Baire's category theorem implies that there exists an index $i$ such that $E_i$ has non-empty relative interior, that is there exists an open set $V$ such that $E\cap V\subset  E_i$.
		There are infinitely many  designated intervals $\widetilde{C}_{N_k}$ contained in $V$ and by the construction of $E$,  each $n_{k+1}$th-level interval in $\mathcal{C}_{n_{k+1}}\cap \widetilde{C}_{N_k}$  intersects $E$. Then
		\begin{align*}
			N_{2^{-n_{k+1}}}(E\cap V)\geq&\  \#(\mathcal{C}_{n_{k+1}}\cap \widetilde{C}_{N_k})\\
			=&\ 2^{n_{k+1}-N_k}\\
			>&\ 2^{n_{k+1}s-1}
		\end{align*} 
by \eqref{nka}.	Hence, $\ubd E_i\geq s$ and so  $\pkd E\geq s$ by \eqref{mpddef}.
		
		By \eqref{m1}-\eqref{less}, for any designated interval $\widetilde{C}_{N_k}$ and $m>N_k$,
	\begin{equation*}
		N_{2^{-m}}(E\cap\widetilde{C}_{N_k})=\#\big(\mathcal{C}_m\cap \widetilde{C}_{N_k}\big)\leq m 2^{ms},
	\end{equation*}
	implying that $\ubd (E\cap\widetilde{C}_{N_k})\leq s$.
	Then
	\begin{equation*}
		\pkd  E=\sup_k \pkd (E\cap\widetilde{C}_{N_k})\leq\sup_k \ubd(E\cap\widetilde{C}_{N_k})\leq s,
	\end{equation*}	
	so	$\pkd  E=s$.

	Let $E_0\subseteq\R$ be any compact set with $\hdd E_0 = \pkd  E_0 =t$. By \eqref{ineqs}
	$t\leq \hdd E_0 \leq  \ucod E_0\leq \pkd  E_0\leq t $.
	 Taking $E$ to be $E\cup E_0$ with $E$ as above in the statement of the proposition gives a set with the desired dimensions.
	\end{proof}

The following proposition combines the previous two results.

\begin{prop}
	Given numbers $0<a<b<c<1$, there exists a compact set $E$ such that $\mlbd E  = a$, $\ucod E=b$ and $\pkd E=c$.
\end{prop}
\begin{proof}
By Proposition \ref{example1} we can find a compact set $E_1$ such that $\mlbd E_1=a$ and $\ucod E_1=\pkd E_1=b$.
By Proposition \ref{propeq}, there exists a compact set $E_2$ such that $\ucod E_2=a$ and $\pkd E_2=c$. The set $E=E_1\cup E_2$ has the required properties.
\end{proof}

\section{Fourier characterisations of dimensions}\label{FourierSec}
\setcounter{equation}{0}
\setcounter{theo}{0}

Characterisation of a fractal dimension in terms of Fourier transforms goes back at least to Kaufman's \cite{Kau} proof of Marstrand's  projection theorem relating to Hausdorff dimension, see also Mattila's books \cite{Mat, Mat2}. However, other dimensions of a set $E$ may be expressed in terms of the Fourier transform of measures on $E$. We define the Fourier transform of a measure  $\mu \in {\mathcal P}(E)$ by 
$$\widehat {\mu}(z) = \int \mathrm{e}^{\mathrm{i}x\cdot z} \mathrm{d}\mu(x)\quad (z\in \mathbb{R}^d).$$
Proposition \ref{fourier1}  relates the behaviour of $\int \mu(B(x,r))d\mu(x)$ of a measure $\mu$ for small $r$ with the mean square of its Fourier transform $\widehat {\mu}$  over large balls. The underlying idea is that $1_{B(0,r)} (x)\approx \exp(-\frac{|x|^2}{2r^2})$ for $x\in \mathbb{R}^d$, and that the  Fourier transform of the Gaussian $\exp(-\frac{|x|^2}{2r^2})$ is $(2\pi)^{d/2} r^d \exp(-\frac{|z|^2 r^2}{2})$ with both of these Gaussians strictly positive. The uniformity of the constants for $\mu \in \mathcal{P}(B(0,\rho))$ is important for applications. 

\begin{prop}\label{fourier1}
Let $0<\epsilon<1$ and $\rho>0$. There are constants $b_1,b_2>0$ and $0<r_0<1 $, depending only on $d, \epsilon$ and $\rho$,    such that for all probability measures $\mu$ on $\mathbb{R}^d$ with support in $B(0,\rho)$ and  all $r \leq r_0$,
\begin{equation}\label{double}
b_1 r^{d(1+\e)} \int_{|z| \leq r^{-1}} |{\widehat {\mu}(z)}|^2 \mathrm{d}z \leq\int \mu(B(x,r))\mathrm{d}\mu(x)
\leq b_2 r^{d(1-\e)} \int_{|z| \leq r^{-1}} |{\widehat {\mu}(z)}|^2 \mathrm{d}z.
\end{equation}
\end{prop}
\begin{proof}

We note  that $1_{B(0,r^{1-\e})}(x)\leq \mathrm{e}^{1/2} \exp(-\frac{|x|^2}{2r^{2(1-\e)}})$ for all $x\in \mathbb{R}^d$. Then for $0<r<1$,
\begin{align}
\int &\mu(B(x,r))\mathrm{d}\mu(x)
 \leq \int \mu(B(x,r^{1-\e}))\mathrm{d}\mu(x)\nonumber \\
 &=\int\!\!\int  1_{B(0,r^{1-\e})} (x-y)\mathrm{d}\mu(x)\mathrm{d}\mu(y)\nonumber \\
 &\leq \mathrm{e}^{1/2} \int\!\!\int  \exp\!\Big(-\frac{|x-y|^2}{2r^{2(1-\e)}}\Big) \mathrm{d}\mu(x)d\mu(y)\nonumber \\
&=\mathrm{e}^{1/2} \int \bigg(\exp\!\Big(-\frac{|\cdot|^2}{2r^{2(1-\e)}}\Big)*\mu\bigg)(y)\mathrm{d}\mu(y)\nonumber \\
&=\frac{\mathrm{e}^{1/2}}{(2\pi)^d}  \int \reallywidehat{\bigg(\exp\!\Big(\displaystyle{ -\frac{|\cdot|^2}{2r^{2(1-\e)}}}\Big)*\mu\bigg)}(z)\overline{\widehat {\mu}(z)}\mathrm{d}z\quad\text{(Plancheral)}\nonumber \\
&=\frac{\mathrm{e}^{1/2}}{(2\pi)^{d}}  \int (2\pi)^{d/2} r^{d(1-\e)} \exp\!\Big(-\frac{|z|^2 r^{2(1-\e)}}{2}\Big){\widehat {\mu}(z)}\overline{\widehat {\mu}(z)} \mathrm{d}z \quad \text{(Parseval + convolution)}\nonumber \\
&=c_1 r^{d(1-\e)}  \int \exp\!\Big(-\frac{|z|^2 r^{2(1-\e)}}{2}\Big)|{\widehat {\mu}(z)}|^2 \mathrm{d}z \qquad \nonumber \\
&\leq c_1 r^{d(1-\e)}  \bigg[\int_{|z| \leq r^{-1}} |{\widehat {\mu}(z)}|^2 \mathrm{d}z +
\int_{|z| > r^{-1}} \exp\!\Big(-\frac{|z|^2 r^{2(1-\e)}}{2}\Big)\mathrm{d}z\bigg]\nonumber \\
&\leq c_1 r^{d(1-\e)} \bigg[\int_{|z| \leq r^{-1}} |{\widehat {\mu}(z)}|^2 \mathrm{d}z + 
c_2\bigg]\label{boundint} \\
&\leq c_3 r^{d(1-\e)} \int_{|z| \leq r^{-1}} |{\widehat {\mu}(z)}|^2 \mathrm{d}z. \label{nearzero}
\end{align}
For \eqref{boundint} we note that the second integral above is bounded in $r$. For \eqref{nearzero}  we note that ${\widehat {\mu}(0)}=1$ and 
$|\nabla  \widehat {\mu}(z)| \leq \sqrt{d}\int_{|x|\leq\rho} |x|\mathrm{d}\mu(x)\leq \sqrt{d}\rho$.
Then $|{\widehat {\mu}(z)}|\geq \frac{1}{2}$  
 if $|z| \leq \frac{1}{2}(\sqrt{d}\rho)^{-1}=:r_0^{-1}$, 
  so 
$$\int_{|z| \leq r^{-1}} |{\widehat {\mu}(z)}|^2 \mathrm{d}z \geq \int_{|z| \leq r_0^{-1} }|{\widehat {\mu}(z)}|^2 \mathrm{d}z \geq  \textstyle{\frac{1}{4}}v_d r_0^{-d}$$ 
if $r\leq r_0$, where $v_d$ is the volume of the $d$-dimensional unit ball.

The left hand inequality of  \eqref{double} follows in a similar way, working with  $\exp(-\frac{|x-y|^2}{2r^{2(1+\e)}})$ and its Fourier transform.
\end{proof}

This proposition allows us to read off Fourier characterisations of various dimensions discussed in Section \ref{sec2}.
\begin{cor}\label{fdimcor}
Let $E \subset \mathbb{R}^d$ be a bounded Borel set. Then
\begin{align*}
\lbd E &=\ \liminf_{R \rightarrow \infty}\frac{\log \Big(R^{-d} \inf_{\mu \in \mathcal{P}(E)} \int_{|z| \leq R} |{\widehat {\mu}(z)}|^2 \mathrm{d}z\Big)}{- \log R}
 \\
\mbox{and} \quad 
\ubd E &=\ \limsup_{R \rightarrow \infty}\frac{\log \Big(R^{-d} \inf_{\mu \in \mathcal{P}(E)} \int_{|z| \leq R} |{\widehat {\mu}(z)}|^2 \mathrm{d}z\Big)}{- \log R}
\end{align*}
\end{cor}
\begin{proof}
These identities follow from Proposition \ref{boxprop} and taking arbitrarily small values of $\epsilon$  and letting $r^{-1}=R \to \infty$ in Proposition \ref{fourier1}, using the uniformity of the inequalities \eqref{double} in $\mu\in\mathcal{P}(E)$.
\end{proof}

Similar results for box dimensions were obtained in \cite{Fr} by a different method.  Another form of Fourier expression for box dimensions is developed in \cite{Fa}.
\begin{cor}
	Let $E \subset \mathbb{R}^d$ be a compact set. Then
	
	\begin{equation*}
		\mlbd E= \sup_{F\subset E}\inf_{\{R_k\}\nearrow\infty}\sup_{\mu:\spt\mu=F}\limsup_{k\to\infty}  \frac{\log \Big(R_k^{-d}   \int_{|z| \leq R_k} |{\widehat {\mu}(z)}|^2 \mathrm{d}z\Big)}{- \log R_k}.  
	\end{equation*}
\end{cor}
\begin{proof}
We apply Proposition \ref{fourier1} to \eqref{MBL} taking arbitrarily small $\epsilon$ and letting $r^{-1}=R \to \infty$.
\end{proof}

There are simple Fourier expressions for the correlation dimensions of measures.
\begin{cor}\label{corcdmes}
For $\mu$ a Borel probability measure with bounded support on $\mathbb{R}^d$,
\begin{align}
\lcod\mu& = \liminf_{R \rightarrow \infty}\frac{\log \Big(R^{-d}  \int_{|z| \leq R} |{\widehat {\mu}(z)}|^2 \mathrm{d}z\Big)}{- \log R}\quad \label{intdef7} \\
\mbox{ and }\quad 
\displaystyle{\ucod } \mu & = \limsup_{R \rightarrow \infty}\frac{\log \Big(R^{-d} \int_{|z| \leq R} |{\widehat {\mu}(z)}|^2 \mathrm{d}z\Big)}{- \log R}.\label{intdef6} 
\end{align}
\end{cor}
\begin{proof}
These expressions again follow by Proposition \ref{fourier1}, applied to \eqref{lcdm} and \eqref{ucdm}. 
\end{proof}

We remark that  when $d=1$, a similar formula for $\lcod\mu$ was also derived in \cite{FNW}.
Taking the supremum of these correlation dimensions over measures supported by a set $E$  leads to Fourier expressions for the correlation dimension of $E$. Together with Proposition \ref{propeq} disproves a conjecture of Fraser \cite{Fr} that the expression in \eqref{intdef6} gives packing dimension.
\begin{cor}
Let $E \subset \mathbb{R}^d$ be a bounded Borel set. Then
\begin{equation*}
\ucod E= \inf\Big\{s \geq 0: \forall  \mu \in {\mathcal P}(E), \text{ for sufficiently large }R,  \int_{|z| \leq R} |{\widehat {\mu}(z)}|^2 \mathrm{d}z \geq R^{d-s } \Big\}. 
\end{equation*}
\end{cor}

\begin{proof}
This combines the definition of upper correlation dimension of $E$ with \eqref{intdef6}.
 \end{proof}

For completeness, we include the Fourier result for the lower correlation dimension, that is Hausdorff dimension.
\begin{cor}\label{haudim}
Let $E \subset \mathbb{R}^d$ be a Borel set. Then
\begin{align}
&\lcod E=\hdd E \nonumber\\
&\quad=\sup\Big\{0\leq s \leq d: \exists\  \mu \in {\mathcal P}(E), \text{ s.t. for sufficiently large }R, \int_{|z| \leq R} |{\widehat {\mu}(z)}|^2 \mathrm{d}z \leq R^{d-s} \Big\}\label{hauf1}\\
&\quad=\sup\Big\{0\leq s \leq d:  \exists\  \mu \in {\mathcal P}(E) \text{ s.t. } 
\int |z|^{s-d}|{\widehat {\mu}(z)}|^2 \mathrm{d}z <\infty\Big\}.\label{hauf2}
\end{align}
\end{cor}

\begin{proof}
Identity \eqref{hauf1} follows from  \eqref{lcdm}, \eqref{haucor} and \eqref{intdef7}, whilst \eqref{hauf2} is the familiar Fourier characterisation of Hausdorff dimension obtained by transforming \eqref{energy}, see \cite{Mat, Mat2}.
\end{proof}

\section*{Acknowledgments}
The authors thank Amlan Banaji for many useful discussions and for pointing out the reference \cite{FNW} and also thank the referee for very helpful comments. SZ has been supported by the New Cornerstone Science Foundation through the New Cornerstone Investigator Program.

\bibliographystyle{plain}

\bigskip
\end{document}